\documentclass[a4paper,11pt]{amsart}
\usepackage{amssymb,amsthm}
\usepackage[OT1]{fontenc}
\usepackage[latin1]{inputenc}
\usepackage{hyperref}
\usepackage[usenames]{color}
\usepackage[curve]{xypic}
\usepackage{enumerate}
\usepackage[active]{srcltx}

\DeclareFontFamily{OT1}{pzc}{}
\DeclareFontShape{OT1}{pzc}{m}{it}%
             {<-> s * [1.100] pzcmi8t}{}
\DeclareMathAlphabet{\mathscr}{OT1}{pzc}%
                                 {m}{it}

\numberwithin{equation}{section}

\def\Cl#1{\ensuremath{{\mathcal {#1}}}}
\def\pv#1{\ensuremath{{\mathsf{#1}}}}
\def\Om#1#2{\ensuremath{\overline\Omega_{#1}{\pv{#2}}}}

\def\malcev{\mathop{\raise1pt\hbox{\footnotesize$\bigcirc$\kern-8pt\raise1pt
      \hbox{\tiny$m$}\kern1pt}}}

\def\op{\lbrack\hskip-1.6pt\lbrack\relax}% old: -1.4pt
\def\cl{\rbrack\hskip-1.6pt\rbrack\relax}%  (idem)

\def\loc1{\ell \pv{I}}
\def\krho{{K(\loc1)}}
\def\expsl{{K(\ell \pv{Sl})}}

\def\dast{\mathbin{**}}

\newtheorem{Thm}{Theorem}[section]
\newtheorem{Prop}[Thm]{Proposition}
\newtheorem{Lemma}[Thm]{Lemma}
\newtheorem{Cor}[Thm]{Corollary}

{
   \theoremstyle{definition}
 
 }

\theoremstyle{definition}
\newtheorem{Example}[Thm]{Example}

\CompileMatrices

\begin{document}

\title
{
  Equidivisible pseudovarieties of semigroups
}
\thanks{Work partially supported respectively by CMUP (UID/MAT/00144/2013)
  and CMUC (UID/MAT/00324/2013), which are funded by FCT (Portugal)
  with national (MEC) and European structural funds (FEDER) under the
  partnership agreement PT2020.
}

\author{Jorge Almeida}
\address{CMUP, Departamento de Matem\'atica,
  Faculdade de Ci\^encias, Universidade do Porto, 
  Rua do Campo Alegre 687, 4169-007 Porto, Portugal.}
\email{jalmeida@fc.up.pt}

\author{Alfredo Costa}
\address{CMUC, Department of Mathematics, University of Coimbra,
  Apartado 3008, EC Santa Cruz, 3001-501 Coimbra, Portugal.}
\email{amgc@mat.uc.pt}

\begin{abstract}
  We give a complete characterization of pseudovarieties of semigroups
  whose finitely generated relatively free profinite semigroups are
  equidivisible. Besides the pseudovarieties of completely simple
  semigroups, they are precisely the pseudovarieties that are closed
  under Mal'cev product on the left by the pseudovariety of locally
  trivial semigroups. A further characterization which turns out to be
  instrumental is as the non-completely simple pseudovarieties that
  are closed under two-sided Karnofsky-Rhodes expansion.
\end{abstract}


\makeatletter
\@namedef{subjclassname@2010}{%
  \textup{2010} Mathematics Subject Classification}
\makeatother

\maketitle

\section{Introduction}

A pseudovariety of semigroups is a class of finite semigroups closed
under taking subsemigroups, homomorphic images and finitary products.
In the past few decades, pseudovarieties provided the main framework
for the research on finite semigroups, motivated by Eilenberg's
correspondence theorem between pseudovarieties and varieties of
languages. In this context, the finitely generated relatively free
profinite semigroups associated to each pseudovariety proved to be of
fundamental importance. We assume the reader has some familiarity with
this background. The books~\cite{Rhodes&Steinberg:2009qt,Almeida:1994a}
are indicated as supporting references.
The paper~\cite{Almeida:2003cshort} might also be useful for someone
looking for a brief introduction.

In this paper we are concerned with equidivisible relatively free
profinite semigroups. A semigroup $S$ is \emph{equidivisible} if for
every $u,v,x,y\in S$, the equality $uv=xy$ implies that $u=x$ and
$v=y$, or that there is $t\in S$ such that $ut=x$ and $v=ty$, or such
that $xt=u$ and $y=tv$. Equidivisible semigroups were introduced
in~\cite{Levi:1944} as a generalization of free semigroups. They were
further investigated in~\cite{McKnight&Storey:1969} where a
characterization of the completely simple semigroups as being a
special class of equidivisible semigroups was given
(cf.~Theorem~\ref{t:McKnight-Storey}).

A pseudovariety of semigroups~$\pv V$
is said to be \emph{equidivisible} if every finitely generated free pro-$\pv V$ semigroup is equidivisible.
The paper~\cite{Almeida&Klima:2015a}
includes results about
a special class of semigroup pseudovarieties, the WGGM pseudovarieties,
the ones whose corresponding relatively free profinite semigroups
are ``weakly generalized group mapping''.
In the same paper, the WGGM condition is applied
to obtain sufficient conditions for a pseudovariety to be finitely join irreducible in the lattice of ordered pseudovarieties.
It is also shown in~\cite{Almeida&Klima:2015a} 
that a sufficient condition for a semigroup pseudovariety $\pv V$
to be WGGM is to be equidivisible and to contain $\pv {LSl}$ (the pseudovariety of finite semigroups whose local monoids are semillatices). 
This adds motivation to investigate the natural
question: which pseudovarieties are equidivisible, other than those of completely simple semigroups?

In this paper we give a complete characterization of the equidivisible pseudovarieties of semigroups (Sec.~\ref{sec:char-equid-pseud}), showing that those which are not contained in
the pseudovariety of completely simple semigroups are precisely the ones of the form $\pv V=\pv {LI}\malcev \pv V$, where $\pv {LI}$ denotes the pseudovariety
of locally trivial semigroups and $\malcev$
denotes the Mal'cev product of pseudovarieties.

The semigroup pseudovarieties of the form $\pv V=\pv A\malcev \pv V$,
where $\pv A$ denotes the pseudovariety of aperiodic semigroups,
are those whose
corresponding variety of languages is closed under language
concatenation~\cite{Straubing:1979a,Chaubard&Pin&Straubing:2006}.
In~\cite[Lemma 4.8]{Almeida&ACosta:2007a} it is shown that every
such pseudovariety is equidivisible, with a proof
that uses the closure under concatenation.
In contrast to the approach made in~\cite{Almeida&ACosta:2007a} for this
class of pseudovarieties,
our proof of the equidivisibility
of the pseudovarieties of the
form $\pv V=\pv {LI}\malcev \pv V$
does not use a characterization in terms of the corresponding
varieties of languages, which are those that are closed under
unambiguous product of languages~\cite{Pin:1980b,Pin:1997}. For our
complete characterization of the equidivisible pseudovarieties, and in
both directions of the proof, we had to use
a distinct approach, based on the two-sided Karnofsky-Rhodes expansion
of semigroups. This approach was inspired by the proof given in~\cite{Rhodes&Steinberg:2002} that if a pseudovariety of semigroups $\pv V$
is stable under (one-sided) right Karnofsky-Rhodes expansion, then
the finitely generated free pro-$\pv V$ semigroups
have unambiguous $\leq_{\mathcal R}$-order.
 It turns out that, by a deep result  of Rhodes \emph{et al.}~\cite{Rhodes&Tilson:1989,Rhodes&Weil:1989b}, the
 pseudovarieties of the form \mbox{$\pv V=\pv {LI}\malcev \pv V$} are
 precisely those that are stable under two-sided Karnofsky-Rhodes
 expansion
 (Corollary~\ref{c:closure-under-Karnofsky-Rhodes-expansion}).
 
 Roughly speaking, the two-sided Karnofsky-Rhodes expansion keeps
 track of the transition edges used to read a word in the two-sided
 Cayley graph of the semigroup expanded. If we only keep track
 of the strongly connected components, we obtain another expansion,
 which we call the \emph{two-sided connected expansion}. We deduce
 from our main result that a pseudovariety is closed under 
 two-sided Karnofsky-Rhodes expansion if and only if it is closed
 under two-sided connected
 expansion~(cf.~Corollary~\ref{c:characterization-of-LImV}).

 Another by-product of our results concerns
 the pseudovariety $\pv {LG}$ of finite semigroups whose local monoids are groups. After showing directly that the equidivisible subpseudovarieties
 of $\pv {LG}$ are precisely those containing $\pv {LI}$, we apply our main result to deduce that,
  whenever $\pv V$ is a subpseudovariety of $\pv {LG}$,
 the join $\pv {LI}\vee \pv V$ is equal to the Mal'cev product
 $\pv {LI}\malcev \pv V$ (Corollary~\ref{c:CSveeLI}).

\section{Preliminaries}
\label{sec:prelims}

Recall that
\cite{Rhodes&Steinberg:2009qt,Almeida:1994a,Almeida:2003cshort} are
our supporting references. We use the standard notation $\pv V\vee\pv
W$, $\pv V*\pv W$, $\pv V\dast \pv W$ and $\pv V\malcev \pv W$
respectively for the join, the semidirect product, the two-sided
semidirect product, and the Mal'cev product of pseudovarieties of
semigroups. Occasionally (only in the preparatory
Section~\ref{sec:two-sided-karnofsky}) we refer to pseudovarieties of
semigroupoids (namely the pseudovariety $\ell {\pv I}$ of trivial
semigroupoids) and to varietal operations concerning them. We refer to
\cite{Rhodes&Steinberg:2009qt,Almeida:2003cshort} for details.

The following
pseudovarieties of semigroups appear in this paper:
\begin{itemize}
\item $\pv S$: finite semigroups;
\item $\pv A$: finite aperiodic semigroups;
\item $\pv G$: finite groups;
\item $\pv I$: trivial semigroups;
\item $\pv D$: finite semigroups all of whose idempotents are right zeros;
\item $\pv K$: finite semigroups all of whose idempotents are left zeros;
\item $\pv {LI}=\pv K\vee\pv D$: finite semigroups whose local monoids are trivial;
\item $\pv {LG}$: finite semigroups whose local monoids are groups;
\item $\pv {LSl}$: finite semigroups whose local monoids are
  semilattices; 
\item $\pv {CS}$: finite completely simple semigroups;
\item $\pv {CR}$: finite completely regular semigroups.
\end{itemize}

For a semigroup $S$, let $S^I$ be the monoid obtained
from $S$ by adding an extra element $I$ which is the
identity of~$S^I$.
This allows
a convenient way of writing the definition of equidivisibility:
the semigroup
$S$ is equidivisible when, for every $u,v,x,y\in S$,
the equality $uv=xy$ implies that there is  $t\in S^I$ such that
$ut=x$ and $v=ty$, or such that $xt=u$ and $y=tv$.

If $f\colon S\to T$ is a semigroup homomorphism,
then we extend $f$ to a homomorphism from $S^I$ to $T^I$, also denoted
$f$, by letting $f(I)=I$.
Note that, for every alphabet~$A$, the monoid $(A^+)^I$
can be identified with the free monoid~$A^\ast$ in a natural manner. In particular,
if $\varphi$ is a homomorphism
from $A^+$ to a semigroup $S$, then we
have a unique extension of $\varphi$ to a homomorphism
from $A^\ast$ to $S^I$, with $\varphi(1)=I$.

We were inspired by~\cite{Rhodes&Steinberg:2002} in the use of
semigroup expansions to obtain our main result. With this reference in
mind (see also \cite{Elston:1999}), we quickly recall that, for a
fixed alphabet~$A$, the category of $A$-generated semigroups is the
category $\Cl S_A$ whose objects are the pairs of the form
$(S,\varphi)$ in which $\varphi$ is an onto homomorphism $A^+\to S$,
and where morphisms $(S,\varphi)\to (T,\psi)$ are the homomorphisms
$\theta\colon S\to T$ such that $\theta\circ\varphi=\psi$. An
\emph{expansion cut to generators} defined in $\mathcal S_A$ is an
endofunctor $F\colon \Cl S_A\to \Cl S_A$ equipped with a natural
transformation from $F$ to the identity functor of~$\Cl S_A$. A
convenient way to refer to $F$ is the notation correspondence
$(S,\varphi)\mapsto (S^F,\varphi^F)$, where the pair $(S^F,\varphi^F)$
is the object $F(S,\varphi)$.

\section{The two-sided Karnofsky-Rhodes expansion}
\label{sec:two-sided-karnofsky}

Let $\varphi$ be a homomorphism from $A^+$ onto a semigroup $S$. The
\emph{two-sided Cayley graph} defined by $\varphi$ is the directed
graph $\Gamma_\varphi$ whose set of vertices is $S^I\times S^I$, and
where an edge from $(s_1,t_1)$ to $(s_2,t_2)$ is a triple
$((s_1,t_1),a,(s_2,t_2))$, with $a\in A$, such that
$s_1\varphi(a)=s_2$ and $t_1=\varphi(a)t_2$. Giving to each edge
$((s_1,t_1),a,(s_2,t_2))$ the label $a$, the graph $\Gamma_\varphi$
becomes a semi-automaton over the alphabet $A$. A labeling of paths is
inherited from the labeling of edges in an obvious way. If $u\in A^+$,
then we denote by $p_{\varphi,u}$, or simply $p_{u}$ if $\varphi$ is
understood, the unique path from $(I,\varphi(u))$ to $(\varphi(u),I)$
labeled by $u$.

For an edge $t$ in a directed graph $H$, we denote by $\alpha(t)$ its source and by
$\omega(t)$ is target. The edge $t$ is a \emph{transition edge} of $H$ if
$\alpha(t)$ and $\omega(t)$ are not in the same strongly connected component
of $H$. Returning our attention to the two-sided Cayley graph $\Gamma_\varphi$,
for a path $p$ in $\Gamma_\varphi$, denote by $T(p)$ the set of
transition edges in $p$.
Let $\equiv_\varphi$ be the binary relation on $A^+$ defined by
$u\equiv_\varphi v$ if and only if
$\varphi(u)=\varphi(v)$ and $T(p_u)=T(p_v)$.
The relation $u\equiv_\varphi v$ is a congruence, a well-known fact
whose routine proof is similar to the explicit proof we give later of an analogous result, Lemma~\ref{l:equidivisible-congruence}.
Denote by $S_\varphi^\krho$ the quotient $A^+/{\equiv_\varphi}$ and by
$\varphi^\krho$ the corresponding quotient homomorphism $A^+\to S_\varphi^\krho$.
For the sake of simplicity, we will write $S^\krho$
instead of $S_\varphi^\krho$, if the dependency on $\varphi$ is implicitly understood.

It is well known that the
correspondence $(S,\varphi)\mapsto (S^\krho,\varphi^\krho)$ is an expansion cut to generators, which is called
the \emph{two-sided Karnofsky-Rhodes expansion}.
There is an alternative way of defining this expansion, which puts it as a special case within a more general framework, and which we refer
briefly, leaving the details for
the bibliographic references supporting our discussion.
The two-sided Karnofsky-Rhodes expansion
is an example of a two-sided semidirect product expansion defined by a variety of
semigroups (in this case, the variety of trivial semigroups), as introduced by Elston in~\cite{Elston:1999}.
In~\cite{Rhodes&Steinberg:2002},
a variation of this approach is followed, one where
pseudovarieties of semigroupoids are used instead of varieties of semigroups.
The notation $S^\krho$ is consistent with the notation used
 in~\cite{Rhodes&Steinberg:2002} for the two-sided semidirect product
 expansion $S^{K(\pv V)}$ of a profinite semigroup $S$ defined by
 a pseudovariety~$\pv V$ of semigroupoids.
 As observed in~\cite[Sec.~10]{Rhodes&Steinberg:2002},
 if $\pv V$ is a pseudovariety of semigroups, then
 $S^{K(\ell\pv V)}$ is the corresponding expansion $S^{K(\pv V)}$ introduced by Elston. 

 Suppose that the alphabet $A$ is finite. If $S$ is finite then $S^\krho$ is finite,
because a kernel class of $\varphi^\krho$
is determined by a kernel class
 of $\varphi$ together with a set of transition edges of
 $\Gamma_\varphi$, and there is only a finite number of such classes
 and edges.

 More generally, as explained in~\cite[Sec.~10]{Rhodes&Steinberg:2002},
 if $S$ is finite and $\pv V$ is a locally finite pseudovariety
 of semigroupoids (which is the case of $\ell\pv {I}$), then
 $S^{K(\pv V)}$ is a finite semigroup which is in a natural way
 a two-sided semidirect product $(\Om {\Gamma_\varphi}V)\dast S$,
 and so if $S$ belongs to a pseudovariety $\pv W$ of semigroups, then
 $S^{K(\pv V)}$ belongs to $\pv V\dast\pv W$.
 
 As remarked at the beginning
 of Section 11 of~\cite{Rhodes&Steinberg:2002},
 the isomorphism
 between the quotient $A^+/{\equiv_\varphi}$
 and the two-sided semidirect product $(\Om {\Gamma_\varphi}\ell \pv{I})\dast S$
 is justified by Tilson's result asserting
 that two paths in a graph $X$ coincide in the 
 locally trivial free category
 generated by $X$ if and only if they have the same
 transition edges.
 
 The following result is a special case
  of~\cite[Theorem 3.6.4]{Rhodes&Steinberg:2009qt}.
 
   \begin{Prop}\label{p:expansion-of-free-prov}
    Let $\pv W$ be a locally finite pseudovariety and let $A$ be a
    finite alphabet. Then $(\Om AW)^\krho$ is
    isomorphic to $\Om A{(\ell\pv {I}\dast\pv W)}$.
  \end{Prop}

It is an easy exercise to show, directly from the definition we gave of
the two-sided Karnofsky-Rhodes expansion,
that if $\pi$ is the canonical projection $S^\krho\to S$
then the semigroup $\pi^{-1}(e)$ satisfies the identity
$xyz=xz$, for every idempotent $e$ of $S$.
Actually, as remarked in the proof
of Lemma~3.4 in \cite{Pin&Straubing&Therien:1988},
we have $\pv {\ell I}\dast\pv V\subseteq \op xyz=xz\cl\malcev \pv V$,
and so $\pv {\ell I}\dast\pv V\subseteq \pv {LI}\malcev \pv V$. 
Denote by $\bigcup_{n\geq 1}\ell{\pv I}\dast^n\pv V$ the
sequence of semigroup pseudovarieties recursively defined by
\begin{equation*}
  \ell{\pv I}\dast^0\pv V=\pv V,\quad
  \ell{\pv I}\dast^n\pv V=\ell{\pv I}\dast(\ell{\pv I}\dast^{n-1}\pv V),\quad n\geq 1.
\end{equation*}
The following theorem is a deep result
  of Rhodes \emph{et~al.}~\cite{Rhodes&Tilson:1989,Rhodes&Weil:1989b}
  which the reader can find in \cite[Corollary 5.3.22]{Rhodes&Steinberg:2009qt}.
 
\begin{Thm}\label{t:iteration-of-semidirect-product}
  Let $\pv V$ be a pseudovariety of semigroups. Then we have
  $\pv {LI}\malcev \pv V=\bigcup_{n\geq 0}\ell{\pv I}\dast^n\pv V$.
\end{Thm}

Say that a pseudovariety of semigroups $\pv V$ is
\emph{closed under the two-sided Karnofsky-Rhodes expansion}
if we have $S_\varphi^\krho\in \pv V$ whenever $S\in\pv V$
and $\varphi$ is a homomorphism from a finitely generated free semigroup onto $S$.
For the reader's convenience, we give a proof of the following easy
consequence of Theorem~\ref{t:iteration-of-semidirect-product}.

  \begin{Cor}\label{c:closure-under-Karnofsky-Rhodes-expansion}
    A pseudovariety of semigroups $\pv V$ is closed under the
    two-sided Karnofsky-Rhodes expansion if and only if
    $\pv V=\pv {LI}\malcev \pv V$.
  \end{Cor}

  \begin{proof}
    Suppose that $\pv V=\pv {LI}\malcev \pv V$.
    Let $S\in \pv V$ and
    consider a homomorphism $\varphi\colon A^+\to S$.  Using the aforementioned fact
    that $S_\varphi^\krho\in \op xyz=xz\cl \malcev \pv V$,
    we immediately get that $S_\varphi^\krho\in \pv V$,
    because $\op xyz=xz\cl\subseteq \pv {LI}$.
    Alternatively, one can apply
    (the easy part of) Theorem~\ref{t:iteration-of-semidirect-product},
    since $S_\varphi^\krho\in \ell \pv I \dast\pv V$.   
    
    Conversely, suppose that $\pv V$ is closed under the
    two-sided Karnofsky-Rhodes expansion. Let $\pv W$ be
    a locally finite subpseudovariety of $\pv V$.
    Then $\Om AW$ belongs to $\pv V$. By hypothesis,
    $(\Om AW)^\krho$ also belongs to $\pv V$.
    Applying Proposition~\ref{p:expansion-of-free-prov},
    we conclude that $\ell{\pv I \dast\pv W}\subseteq \pv V$.
    As $\pv W$ can be any
  locally finite subpseudovariety
  of $\pv V$, we actually have
  $\ell{\pv I \dast\pv V}\subseteq \pv V$.
  We deduce from
  Theorem~\ref{t:iteration-of-semidirect-product}
  that $\pv {LI}\malcev \pv V=\pv V$.
  \end{proof}
  
\section{The two-sided connected expansion}

In this section we show that the pseudovarieties closed under the
two-sided Karnofsky-Rhodes expansion are equidivisible. Actually, it
is not necessary to use the full force of the definition of the
expansion. It suffices to use a weaker expansion which we introduce in
this section.

Let $\varphi$ be a homomorphism from $A^+$ onto
a semigroup $S$.
Given a path $p$ in the two-sided Cayley graph $\Gamma_\varphi$,
denote by $C(p)$ the set of
strongly connected components of $\Gamma_\varphi$
that contain some vertex in $p$.
Let $\approx_\varphi$ be the binary relation on $A^+$ defined by
$u\approx_\varphi v$ if and only if
$\varphi(u)=\varphi(v)$ and $C(p_u)=C(p_v)$.

\begin{Lemma}\label{l:equidivisible-congruence}
  The relation $\approx_\varphi$ is a congruence.
\end{Lemma}

\begin{proof}
  The relation $\approx_\varphi$ is clearly an equivalence.
  Taking into account the symmetry of its definition,
  to prove that $\approx_\varphi$ is a congruence, it suffices to show
  that $C(p_{uw})=C(p_{vw})$ whenever $u,v,w\in A^+$ are such
  that $u\approx_\varphi v$. Let $x$ be a vertex of $p_{uw}$
  that is not in $p_{vw}$, where $u\approx_\varphi v$. Then, as $\varphi(u)=\varphi(v)$, we necessarily have
  $x=(\varphi(u_1),\varphi(u_2w))$ for some $u_1,u_2\in A^+$ such that
  $u=u_1u_2$. Since $x'=(\varphi(u_1),\varphi(u_2))$ is a vertex of
  $p_u$, there is some vertex $x''$ in $p_v$ such that $x'$ and $x''$ are in the same strongly connected component.
  Let $t$ be the label of a path from $x'$ to $x''$, and let $z$
  be the label of a path from $x''$ to $x'$.
  We have $x''=(\varphi(v_1),\varphi(v_2))$
  for some $v_1,v_2\in A^\ast$ such that $v=v_1v_2$.
    Note that
  \begin{equation}\label{eq:tz}
    \varphi(u_1t)=\varphi(v_1),\;
    \varphi(u_2)=\varphi(tv_2),\;
    \varphi(v_1z)=\varphi(u_1),\;
    \varphi(v_2)=\varphi(zu_2).
  \end{equation}
  Consider the vertex $x'''=(\varphi(v_1),\varphi(v_2w))$.
  Looking at~\eqref{eq:tz}, we see that we also have
  $\varphi(u_2w)=\varphi(tv_2w)$
  and $\varphi(v_2w)=\varphi(zu_2w)$,
   whence there is a path from $x$ to $x'''$ labeled $t$,
  and a path from $x'''$ to $x$ labeled~$z$. This shows that $x'''$
  belongs to the strongly connected component of $x$. Since $x'''$
  belongs to $p_{vw}$, this establishes the inclusion
  $C(p_{uw})\subseteq C(p_{vw})$. Dually, we
  have $C(p_{vw})\subseteq C(p_{uw})$.
\end{proof}

   If $S$ is finite then $\approx_\varphi$ has finite index,
 because a $\approx_\varphi$-class is defined by a kernel class
 of $\varphi$ together with a set of strongly connected components
 of~$\Gamma_\varphi$, and there is only a finite number of such classes
 and components.

We denote by $S^C$ the quotient $A^+/{\approx_\varphi}$ and by
$\varphi^C$ the canonical homomorphism $A^+\to S^C$.
Clearly, if $\varphi^\krho(u)=\varphi^\krho(v)$ then
$\varphi^C(u)=\varphi^C(v)$, and so
$S^C$ is a quotient of $S^\krho$. The semigroups
$S^C$ and $S^\krho$ may not be isomorphic. For example, consider
the onto homomorphism $\varphi$ from the
two-letter alphabet $A=\{a,b\}$ onto the trivial semigroup $S$.
Then $\varphi^C(a)=\varphi^C(b)$ but $\varphi^\krho(a)\neq \varphi^\krho(b)$.

\begin{Prop}
  The correspondence
  $(S,\varphi)\mapsto (S^C,\varphi^C)$ is an expansion cut to generators.
\end{Prop}

\begin{proof}
  Let $\varphi\colon A^+\to S$ and $\psi\colon A^+\to T$ be onto
  homomorphisms, and let $f\colon S\to T$ be a homomorphism such that
  $f\circ\varphi=\psi$.
  The mapping $f$ induces a homomorphism $\bar f$ of semi-automata
  from $\Gamma_\varphi$ to $\Gamma_\psi$, defined by the following
  mappings from vertices and edges of~$\Gamma_\varphi$ respectively to
  vertices and edges of~$\Gamma_\psi$:
  \begin{equation*}
        (s,t)\mapsto (f(s),f(t)),\qquad
    ((s,t),a,(s',t'))\mapsto ((f(s),f(t)),a,(f(s'),f(t'))).
  \end{equation*}
  Let $u,v\in A^+$ be such that $u\approx_\varphi v$.
  Then we have $\varphi(u)=\varphi(v)$ and $\psi(u)=\psi(v)$.
  Let $y=(\psi(u_1),\psi(u_2))$ be a vertex in $p_u^\psi$,
  where
  $u=u_1u_2$, with $u_1,u_2\in A^\ast$. Then $y=\bar f(x)$,
  where $x=(\varphi(u_1),\varphi(u_2))$ is a vertex in $p_u^\varphi$.  
  Since $u\approx_\varphi v$, there is a vertex $x'$
  in  $p_v^\varphi$ such that $x$ and $x'$
  are in the same strongly connected component
  of $\Gamma_\varphi$.
  Clearly, $y=\bar f(x)$ and $y'=\bar f(x')$
  are in the same strongly connected component of $\Gamma_\psi$.
  As $x'$ is in $p_v^\varphi$,
  we have $y'$ in $p_v^\psi$, showing that
  $C(p_u^\psi)\subseteq C(p_v^\psi)$.
  By symmetry, we have $C(p_u^\psi)\supseteq C(p_v^\psi)$.
  This establishes the equality $u\approx_\psi v$, and therefore
  we may consider the unique semigroup
  homomorphism $f^C\colon S^C\to T^C$
  such that $f^C\circ \varphi^C=\psi^C$.
\end{proof}

Let us call the expansion $(S,\varphi)\mapsto
(S^C,\varphi^C)$ the \emph{two-sided connected expansion}.

A pseudovariety of semigroups $\pv V$
is \emph{closed under two-sided connected expansion}
if we have $S^C$ whenever $S\in\pv V$.

\begin{Prop}\label{p:equidivisibility-under-connected-expansion}
  If\/ $\pv V$ is a pseudovariety of semigroups closed under two-sided
  connected expansion then $\pv V$ is equidivisible.
\end{Prop}

\begin{proof}
    Let $A$ be a finite alphabet. Suppose $u,v,x,y$ are elements
  of $\Om AV$ such that $uv=xy$. Let $\Phi$ be a continuous 
  homomorphism onto a semigroup~$S$ from~$\pv V$,
  and let $\varphi$ be its restriction
  to~$A^+$.
  Denote by $\Phi^C$ the unique continuous homomorphism
  from \/$\Om AV$ onto $S^C$
  whose restriction to $A^+$ coincides with
  $\varphi^C$.
  Consider elements $u_0,v_0,x_0,y_0$ of $A^+$ such that
  $\varphi^C(u_0)=\Phi^C(u)$,
  $\varphi^C(v_0)=\Phi^C(v)$,
  $\varphi^C(x_0)=\Phi^C(x)$
  and $\varphi^C(y_0)=\Phi^C(y)$.
  The vertex $(\Phi(u),\Phi(v))=(\varphi(u_0),\varphi(v_0))$ belongs
  to the path
  $(I,\varphi(u_0v_0))\xrightarrow{u_0v_0}(\varphi(u_0v_0),I)$
  of the two-sided Cayley graph~$\Gamma_\varphi$.
  Since $uv=xy$, we know that
  $\Phi^C(uv)=\Phi^C(xy)$. Therefore, there is
  a vertex $(r,s)$ in the path
  $(I,\varphi(u_0v_0))\xrightarrow{x_0y_0}(\varphi(u_0v_0),I)$
  of $\Gamma_\varphi$ which lies in the strongly connected component
  of $(\Phi(u),\Phi(v))$. Since $(\Phi(x),\Phi(y))$ is clearly also
  in the path
  $(I,\varphi(u_0v_0))\xrightarrow{x_0y_0}(\varphi(u_0v_0),I)$,
  we conclude that in $\Gamma_\varphi$ there is a (possibly empty) path from
  $(\Phi(u),\Phi(v))$ to $(\Phi(x),\Phi(y))$ or a path
  from $(\Phi(x),\Phi(y))$ to $(\Phi(u),\Phi(v))$.

  Therefore, there is a word $t_\Phi\in A^\ast$ such that
  \begin{equation*}
    \begin{cases}
      \Phi(ut_\Phi)=\Phi(x)\\
      \Phi(v)=\Phi(t_\Phi y),
    \end{cases}
  \end{equation*}
  in which case we say that $\Phi$ is \emph{of the first type}, or
  there is a word $\tau_\Phi\in A^\ast$
  such that
  \begin{equation*}
        \begin{cases}
      \Phi(x\tau_\Phi)=\Phi(u)\\
      \Phi(y)=\Phi(\tau_\Phi v),\\
    \end{cases}
  \end{equation*}
  in which we say that $\Phi$ is \emph{of the second type}.
  Note that $\Phi$ can be simultaneously of the first type and of the second type.
  The result now follows from a standard argument, which we write down
  for the reader's convenience. We know that $\Om AV$ is the inverse limit
  of an inverse system of
  semigroups from~$\pv V$ defined by a countable set of connecting
  onto homomorphisms of the form $\pi_{m,n}\colon S_m\to S_n$,
  where $m,n$ are arbitrary positive integers with $m\geq n$.
  For each $n\geq 1$,
  let $\pi_n$ be the projection $\Om AV\to S_n$
  associated to this inverse system.
  Note that, for each $n\leq m$,
  the homomorphism $\pi_n$ is of the same type as $\pi_m$.
  On the other hand, since there are only two types,
  at least one of them occurs infinitely often.
  Combining these two simple observations we
  conclude that
  $\pi_n$ is of the first type
  for every $n\geq 1$, or of the second type
    for every $n\geq 1$.
  Without loss of generality, we assume the former case.
  Denote $t_{\pi_n}$ by $t_n$. We have
    \begin{equation}\label{eq:type-1}
      \pi_n(ut_n)=\pi_n(x)\quad\text{and}\quad
      \pi_n(v)=\pi_n(t_ny),
  \end{equation}
  for every $n\geq 1$.  
  Let $t$ be an accumulation
  point in $(\Om AV)^I$ of the sequence $(t_n)_{n}$.
  Fix $k\geq 1$, and let $n\geq k$.
  Applying $\pi_{n,k}$ to~\eqref{eq:type-1},
  we get
    \begin{equation*}
      \pi_k(ut_n)=\pi_k(x)\quad\text{and}\quad
      \pi_k(v)=\pi_k(t_ny),
  \end{equation*}
  for every $n\geq k$. By continuity of $\pi_k$,
  we obtain
      \begin{equation*}
      \pi_k(ut)=\pi_k(x)\quad\text{and}\quad
      \pi_k(v)=\pi_k(ty).
    \end{equation*}
    This implies $ut=x$ and $v=ty$.
    The case where
    $\pi_n$ is of the second type for every $n\geq 1$
    leads to the existence of $\tau$ in
    $(\Om AV)^I$ such that $x\tau=u$ and $y=\tau v$.
\end{proof}

\begin{Cor}\label{c:equidivisibility-under-connected-expansion}
  If\/ $\pv V$ is a pseudovariety of semigroups closed under 
  two-sided Karnofsky-Rhodes expansion then $\pv V$ is equidivisible.
\end{Cor}

\begin{proof}
  It follows immediately from Proposition~\ref{p:equidivisibility-under-connected-expansion}
  and the fact that the two-sided connected expansion of a semigroup $S$ is a homomorphic image
  of the two-sided Karnofsky-Rhodes expansion of $S$.
\end{proof}

\begin{Cor}\label{c:equidivisibility-of-LImV}
  If $\pv V$ is a pseudovariety of semigroups such that
  $\pv V=\pv {LI}\malcev \pv V$ then $\pv V$ is equidivisible.
\end{Cor}

\begin{proof}
  Apply
  Corollary~\ref{c:equidivisibility-under-connected-expansion}
  and (the easy part of) Corollary~\ref{c:closure-under-Karnofsky-Rhodes-expansion}.
\end{proof}

\section{Equidivisible subpseudovarieties of \protect{\pv {CR}}}

It was proved in~\cite{McKnight&Storey:1969}
that every completely simple semigroup is
equidivisible. In fact, the following
stronger result was established.

\begin{Thm}\label{t:McKnight-Storey}
  A semigroup $S$ is completely simple if and only if for every
  $u,v,x,y\in S$, the equality $uv=xy$ implies the existence of
  $t_1,t_2\in S^I$ such that $ut_1=x$, $t_1y=v$, $u=xt_2$, and
  $y=t_2v$.
\end{Thm}

On the other hand we have the following simple fact.

\begin{Lemma}\label{l:equidivisible-within-CR}
    If $\pv V$ is a pseudovariety of completely regular semigroups
  containing $\pv{Sl}$ then 
  $\pv V$ is not equidivisible.
\end{Lemma}

\begin{proof}
    Consider the alphabet $A=\{a,b\}$.
  We claim that $\Om AV$ is not equidivisible.
  Indeed, we have $ab\cdot (ab)^\omega=a\cdot b$.
  On the other hand, since $c(ab)\nsubseteq c(a)$, there is no
  $t\in (\Om AV)^I$ such that $ab\cdot t=a$. Similarly, there is no
  $t\in (\Om AV)^I$ such that $t\cdot (ab)^\omega=b$. This establishes
  the claim.
\end{proof}

Since a pseudovariety of completely regular semigroups not containing
$\pv {Sl}$ is contained in $\pv {CS}$,
combining Theorem~\ref{t:McKnight-Storey} and
Lemma~\ref{l:equidivisible-within-CR} we deduce the following result.

\begin{Cor}\label{c:equidivisible-pseudovarieties-in-CR}
  A pseudovariety of completely regular semigroups is equidivisible if and only if it is contained in $\pv {CS}$.\qed
\end{Cor}

\section{Letter super-cancelability as a necessary condition for equidivisibility}

Let $S$ be an $A$-generated compact semigroup.
This implies $S=S^IA=AS^I$.
Say that $S$ is \emph{right letter cancelative}
when, for every $a\in A$ and $u,v\in S^I$, 
the equality $ua=va$ implies
$u=v$.
Say moreover that $S$ is \emph{right letter super-cancelative}
when, for every $a,b\in A$ and $u,v\in S^I$, 
 the equality $ua=vb$ implies
 $a=b$ and $u=v$. We have the obvious
 dual notions of \emph{left letter cancelative} and
 \emph{left letter super-cancelative} semigroup.
 If $S$ is simultaneously right and left letter (super-)cancelative,
then we say $S$ is \emph{letter (super-)cancelative}.

Say that a pseudovariety of semigroups $\pv V$ is right letter (super-)cancelative if
$\Om AV$ is right letter (super-)cancelative, for every finite alphabet $A$. One also has the dual notions of left letter (super-)cancelative and letter
(super-) cancelative pseudovariety.

\begin{Example}\label{eg:stabilized-by-D-are-fin-cancel}
If\/ $\pv V$ is a semigroup pseudovariety containing some nontrivial
monoid
and such that $\pv V=\pv V\ast\pv D$,
then $\pv V$ is letter super-cancelative (cf.~\cite[Prop.~1.60]{ACosta:2007t}
and~\cite[Exercise 10.2.10]{Almeida:1994a}).
\end{Example}

In~\cite{Almeida&Klima:2015a} one finds a characterization of the
(right/left) super-cancelative pseudovarieties
as a routine exercise of application of basic results in
the theory of profinite semigroups. The following simple observation
is included in that characterization.

\begin{Lemma}\label{l:cancelable-must-contain}
  A semigroup pseudovariety $\pv V$ is right letter super-cancelative if and only if $\pv D\subseteq \pv V$ and $\pv V$ is right letter cancelative.
  Dually, $\pv V$ is left letter super-cancelative if and only if
  $\pv K\subseteq \pv V$ and $\pv V$ is left letter cancelative.
  Therefore, $\pv V$ is letter super-cancelative if and only if
  $\pv {LI}\subseteq\pv V$ and $\pv V$ is letter cancelative.
\end{Lemma}

Letter super-cancelability appears as a necessary condition for equidivisibility in the following way. 

\begin{Prop}\label{p:equid-conca-are-finitely-cancelable}
  If \/$\pv V$ is an equidivisible pseudovariety of semigroups
  not contained in $\pv {CS}$ then
  $\pv V$ is letter supper-cancelative.
\end{Prop}

\begin{proof}
  Fix a finite alphabet $A$.
  Let $u,v\in(\Om AS)^I$ and let $a,b\in A$
  be such that $\pv V\models ua=vb$.
  Since $\Om AV$ is equidivisible, there is $t\in (\Om AV)^I$ such
  that $\pv V\models ut=v$ and $\pv V\models a=tb$,
  or such that $\pv V\models vt=u$ and
  $\pv V\models b=ta$.
  Suppose that $t\neq I$. Suppose also that $\pv V\models a=tb$.
  Replacing by $a$ every letter occurring in $tb$,
  we get $\pv V\models a=a^\nu$ for some profinite exponent $\nu>1$.
  This implies that $\pv V\subseteq \pv {CR}$.
  Corollary ~\ref{c:equidivisible-pseudovarieties-in-CR}
  states that the equidivisible subpseudovarieties of $\pv {CR}$
  are precisely the subpseudovarieties of $\pv {CS}$.
  Since $\pv V$ is not contained in $\pv {CS}$,
  we reach a contradiction.
  Similarly, $\pv V\models a=tb$ leads to a contradiction.
  To avoid the contradiction, we must have $t=I$,
  whence $a=b$ and $\pv V\models u=v$.
  Symmetrically, $\pv V\models au=bv$ implies $a=b$ and
  $\pv V\models u=v$.
\end{proof}

Combining Proposition~\ref{p:equid-conca-are-finitely-cancelable}
with Lemma~\ref{l:cancelable-must-contain} we get the following corollary.

\begin{Cor}\label{c:equidiv-not-in-CR}
  If \pv V is an equidivisible pseudovariety not contained in~\pv{CS},
  then \pv V contains~\pv{LI}.\qed
\end{Cor}

\section{Equidivisible subpseudovarieties of \protect{\pv {LG}}}

Recall that the class \pv{LG} of all finite local groups is the largest
pseudovariety of semigroups whose semilattices are trivial. Local
groups are thus generalizations of completely simple semigroups that
turn out to be sometimes much harder to handle. For our purposes, the
following technical result turns out to be essential.

\begin{Lemma}\label{l:DinVinLG}
  If\/ \pv V is a subpseudovariety of\/ \pv{LG} containing~\pv D, then \pv
  V is right letter super-cancelative.
\end{Lemma}

\begin{proof}
  Let $u$ and $v$ be pseudowords and $a$ and $b$ be letters such that
  the pseudoidentity $ua=vb$ holds in~\pv V.
  The pseudovariety  \pv D is right
  letter super-cancelative, whence $a=b$ and \pv D satisfies $u=v$.
  In particular, $u=v$ or both $u$  and $v$ are infinite pseudowords.
  Suppose the latter case occurs.
  Let $s_n$ be the suffix of~$u$ of length~$n$, which is also the
  suffix of length~$n$ of~$v$, that is, there are factorizations
  $u=u_ns_n$ and $v=v_ns_n$. By compactness, there exists a convergent
  subsequence of the sequence of triples $(u_n,v_n,s_n)$ and,
  therefore, there exist pseudowords $u',v',w$, where $w$ is infinite,
  such that $u=u'w$ and $v=v'w$ (in~\pv S). Since $w$~is infinite,
  there is a factorization of the form $w=w_1w_2^\omega w_3$
  \cite[Corollary~5.6.2]{Almeida:1994a}. As \pv V~is contained
  in~\pv{LG}, it must satisfy the following pseudoidentities:
  $$w=w_1w_2^\omega w_3
  =w_1(w_2^\omega w_3aw_2^\omega)^\omega w_3
  =w_1w_2^\omega w_3at
  =wa\cdot t
  $$
  where $t=(w_2^\omega w_3aw_2^\omega)^{\omega-1}w_3$.
  Hence, since the pseudovariety \pv V satisfies $ua=va$, it also
  satisfies
  $$u=u'w
  =u'wa\cdot t
  =ua\cdot t
  =va\cdot t
  =v'wa\cdot t
  =v'w
  =v.\popQED
  $$
\end{proof}

\begin{Thm}\label{t:equidiv-in-LG}
  A subpseudovariety of\/ \pv{LG} is equidivisible if and only if it is
  contained in~\pv{CS} or it contains~\pv{LI}.
\end{Thm}

\begin{proof}
  We already know that every subpseudovariety of~\pv{CS} is
  equidivisible by~Theorem \ref{t:McKnight-Storey}.
  Thus, for the remainder of the proof, we assume that
  \pv V is a subpseudovariety of~\pv{LG} not contained in~\pv{CS}.

  If \pv V is equidivisible, then \pv V contains~\pv{LI} by
  Corollary~\ref{c:equidiv-not-in-CR}. For the converse, suppose that \pv
  V contains~\pv{LI} and that \pv V satisfies the pseudoidentity
  $uv=xy$. By Lemma~\ref{l:DinVinLG}, \pv V is right letter
  super-cancelative. By duality, \pv V~is also left letter super-cancelative.
  Therefore, to prove that there is a (possibly empty) pseudoword $t$
  such that in $\pv V$  we have $ut=x$ and $v=ty$, or $xt=u$ and $y=tv$, we may assume that all the pseudowords
  $u,v,x,y$ are infinite.
  Hence, $uv=xy$ may be viewed as an equality
  between products in the minimum ideal of a suitable \Om AV, which
  is a completely simple semigroup. By
  Theorem~\ref{t:McKnight-Storey}, we may conclude that \pv V is
  equidivisible.
\end{proof}

\section{Characterization of equidivisible pseudovarieties}
\label{sec:char-equid-pseud}

Let $\varphi$ be a homomorphism from $A^+$ onto a semigroup $S$.
Given $u\in A^+$, a \emph{transition edge for $u$ in $\Gamma_\varphi$} is an element of
$T(p_u)$.
Note that $T(p_u)$ is always nonempty, since there is no path from
$(\varphi(u),I)$ to $(I,\varphi(u))$ in $\Gamma_\varphi$.
If $A$ and $S$ are finite, then $T(p_u)$ is finite and so, for some integer $n$,
we can consider the \emph{sequence $(\varepsilon_i)_{i\in \{1,\ldots,n\}}$ of transition edges for $u$ in $\Gamma_\varphi$}, where $\varepsilon_i$ is the $i$-th transition edge appearing in $p_u$.

In this section we shall work primarily with the expansion
$\varphi^\krho$, but at some point it will be
convenient to use another expansion which we next describe.
In a graph, the \emph{content} of a path $p$
is the set $c(p)$ of edges in the path.
Consider the relation $\equiv_{\varphi,\pv {Sl}}$ on $A^+$
defined by $u\equiv_{\varphi,\pv {Sl}} v$ if and only if
$c(p_u)=c(q_v)$.
Note that since $p_w$ starts at $(I,\varphi(w))$,
if $u\equiv_{\varphi,\pv {Sl}} v$ then
$\varphi(u)=\varphi(v)$ holds. The relation $\equiv_{\varphi,\pv {Sl}}$
is a congruence and the quotient  homomorphism $A^+\to A^+/{\equiv_{\varphi,\pv {Sl}}}$
is precisely the two-sided semidirect product expansion
$\varphi^\expsl\colon A^+\to S^\expsl$
(cf.~ \cite[Sec.~5.4]{Elston:1999}).\footnote{The reader is cautioned for some misprints in the discussion made in \cite[Sec.~5.4]{Elston:1999}; for instance, at some point a map $\psi$ is defined that takes a word $u$ to the set of edges in $p_u$, and not to the set of transition edges as by mistake it is written there. The characterization of $\varphi^\expsl$ is an application
  of \cite[Corollary 5.4]{Elston:1999} and of a result of I.~Simon~\cite{Simon:1972} (a proof of which may be found in \cite[Theorem VIII.7.1]{Eilenberg:1976})
  stating that two paths in a graph have the same content if and only if they are equal in the free
  category relatively to $\ell\pv {Sl}$.}

Suppose moreover that $A$ and $S$ are finite.
Then $S^\krho$ and $S^\expsl$
are both finite semigroups.
Denote by $\Phi$ (respectively, $\Phi^\krho$ and $\Phi^\expsl$)
the unique continuous homomorphism from $\Om AS$ onto $S$
(respectively, $S^\krho$ and $S^\expsl$)
whose restriction to $A^+$ is $\varphi$
(respectively, $\varphi^\krho$ and $\varphi^\expsl$).
Let $u\in\Om AS$. Consider an arbitrary sequence
$(u_n)_n$ of elements of $A^+$ converging to $u$. Then, there is $N$ such that
$\Phi^\krho(u)=\varphi^\krho(u_n)$
for every $n\geq N$.
Therefore, we can define a \emph{transition edge for $u$ in $\Gamma_\varphi$}
as being an element of $T(p_{u_n})$ for every sufficiently large $n$, since this set depends only on $\varphi$ and~$u$. In a similar way,
one can define
the \emph{sequence of transition edges for $u$ in~$\Gamma_\varphi$}
as being the sequence of transition edges for $u_n$ in~$\Gamma_\varphi$
for every sufficiently large $n$,
and an \emph{edge for $u$ in $\Gamma_\varphi$}
as being an element of $c(p_{u_n})$ for every sufficiently large $n$.
Note that a transition edge for $u$ is indeed an edge for~$u$.

If $\psi$ is a continuous homomorphism from $\Om AS$ onto a finite semigroup~$T$, then we denote by $\Gamma_\psi$ the two-sided Cayley graph of the restriction of $\psi$ to~$A^+$. Since $\psi$ is the unique continuous extension
to~$\Om AS$ of its restriction to~$A^+$, the homomorphisms $\psi^\krho$
and $\psi^\expsl$ are defined in view of the previous paragraph.
Their images are also denoted $S_\varphi^\krho$ and $S_\psi^\expsl$, respectively.  

  \begin{Lemma}\label{l:factorization-around-a-transition-edge}
    Let $\varphi$ be a continuous homomorphism from $\Om AS$
    onto a finite semigroup $S$, where $A$ is a finite alphabet.
    Let $u\in\Om AS$.
    If $((s_1,t_1),a,(s_2,t_2))$ is an edge for $u$ in~$\Gamma_\varphi$,
    then there is a factorization
    $u=u_1au_2$ of $u$, with $u_1,u_2\in (\Om AS)^I$,
    such that $\varphi(u_1)=s_1$ and $\varphi(u_2)=t_2$.
  \end{Lemma}

  \begin{proof}
    We may consider a sequence $(u_n)_n$ of elements of $A^+$
    converging to $u$ and
    such that $\varphi^\expsl(u_n)=\varphi^\expsl(u)$ for every $n$.
    In particular, for every~$n$, the edge $((s_1,t_1),a,(s_2,t_2))$
    is an edge for $u_n$ in~$\Gamma_\varphi$, and so $u_n$ factors
    as $u_n=u_{n,1}au_{n,2}$ for some $u_{n,1},u_{n,2}\in A^\ast$
    such that $\varphi(u_{n,1})=s_1$ and $\varphi(u_{n,2})=t_2$.
    By compactness, the sequence of pairs $(u_{n,1},u_{n,2})$
    has some accumulation
    point $(u_1,u_2)$ in $(\Om AS)^I\times (\Om AS)^I$.
    By continuity of multiplication and of $\varphi$,
    we have $u=u_1au_2$, $\varphi(u_1)=s_1$ and $\varphi(u_2)=t_2$.    
  \end{proof}

  \begin{Lemma}\label{l:letter-augmentation}
    Let $\theta\colon B^+\to A^+$ be a homomorphism
    satisfying $\theta(B)=A$,
    for some alphabets $A$ and $B$.
    Consider a homomorphism $\varphi$ from $A^+$ onto a semigroup $S$.
    Let $\psi$ be the homomorphism from $B^+$ onto $S$
    such that $\psi=\varphi\circ\theta$.
    Then we have
    \begin{equation*}
      \psi^\krho(u)=\psi^\krho(v)\implies
      \varphi^\krho(\theta(u))=\varphi^\krho(\theta(v)),
    \end{equation*}
    for every $u,v\in B^+$. Consequently,
    $S_\varphi^\krho$ is a homomorphic image of $S_\psi^\krho$.
  \end{Lemma}

  \begin{proof}
    The proof relies on the following fact:
    there is in $\Gamma_\psi$ a path from
    $(s_1,t_1)$ to $(s_2,t_2)$ labeled by $u\in B^+$
    if and only if
    there is in $\Gamma_\varphi$ a path from
    $(s_1,t_1)$ to $(s_2,t_2)$ labeled by $\theta(u)\in A^+$.
    In particular, $((s_1,t_1),b,(s_2,t_2))$
    is a transition edge of $\Gamma_\psi$
    if and only if
    $((s_1,t_1),\theta(b),(s_2,t_2))$
    is a transition edge of $\Gamma_\varphi$.

    Let $u,v\in B^+$ be such that
    $\psi^\krho(u)=\psi^\krho(v)$.
    Then we immediately have
    $\varphi(\theta(u))=\varphi(\theta(v))$. Suppose that
    $((s_1,t_1),a,(s_2,t_2))$
    is a transition edge 
    of $\Gamma_\varphi$ belonging to the path $p_{\theta(u)}$.
    We have $s_1=\varphi(w_1)$
    and $t_2=\varphi(w_2)$ for some $w_1,w_2\in A^\ast$
    such that $\theta(u)=w_1aw_2$.
    There is a factorization $u=u_1bu_2$
    with $w_1=\theta(u_1)$,
    $w_2=\theta(u_2)$ and
    $a=\theta(b)$. It then follows that
    $((s_1,t_1),b,(s_2,t_2))$
    is a transition edge of $\Gamma_\psi$ belonging to $p_{u}$.
    It belongs also to $p_{v}$, since
    $\psi^\krho(u)=\psi^\krho(v)$.
    Therefore,
    $((s_1,t_1),a,(s_2,t_2))$ is a transition edge of $\Gamma_\varphi$ in $p_{\theta(v)}$.
    Symmetrically, every transition edge
    of $\Gamma_\varphi$
    belonging to $p_{\theta(v)}$
    also belongs to $p_{\theta(u)}$, establishing that 
    $\varphi^\krho(\theta(u))=\varphi^\krho(\theta(v))$.

    Therefore, we can consider the onto homomorphism
    $\rho\colon S_\psi^\krho\to S_\varphi^\krho$ defined by
    $\rho(\psi^\krho(u))=\varphi^\krho(\theta(u))$.    
  \end{proof}
  
We are now ready to prove our main result.  
  
\begin{Thm}\label{t:complete-characterization-of-equidivisible-pseudovarieties}
  A pseudovariety of semigroups $\pv V$
  is equidivisible if and only if it is contained in $\pv {CS}$ or
 it is closed under the two-sided Karnofsky-Rhodes expansion.
\end{Thm}

\begin{proof}
  The ``if'' part follows from Corollaries~\ref{c:equidivisibility-under-connected-expansion} and~\ref{c:equidivisible-pseudovarieties-in-CR}.

  Conversely, suppose that $\pv V$ is an equidivisible pseudovariety
  not contained in $\pv {CS}$.
  Let $u,v$ be elements
  of $\Om XS$ such that $\pv V\models u=v$, where $X$ is some finite alphabet.
  For a finite alphabet $A$, let $\varphi$ be a continuous homomorphism
  from $\Om AS$ onto a semigroup $S$ from~$\pv V$.
  Consider the finite alphabet $B=A\cup X$.
  There is a continuous onto homomorphism $\theta\colon \Om BS\to \Om AS$
  such that $\theta(B)=A$.
  Let $\psi$ be the unique continuous homomorphism from $\Om BS$ onto
  $S$ such that $\psi=\varphi\circ\theta$.
  We will show that $S_\psi^\krho\models u=v$.

  Viewing $u,v$ as elements of $\Om BS$,
  and because $\pv V\models u=v$, we have $\psi(u)=\psi(v)$.
  We claim that
  $\psi^\krho(u)=\psi^\krho(v)$.
  Let $(\varepsilon_i)_{i\in \{1,\ldots,n\}}$
  and $(\delta_i)_{i\in \{1,\ldots,m\}}$ be 
  the sequences of transition edges in $\Gamma_{\psi}$ respectively
  for $u$ and for $v$.
  Without loss of generality, we may assume that~$n\leq m$.

  Suppose that the set
  \begin{equation}\label{eq:characterization-0}
       \{i\in \{1,\ldots,n\}\mid \varepsilon_i\neq \delta_i\}
  \end{equation}
  is nonempty, and let
  $j$ be its minimum.
  By Lemma~\ref{l:factorization-around-a-transition-edge},
  there are factorizations $u=u_1au_2$ of $u$
  and $v=v_1bv_2$ of $v$, with $a,b\in B$ and $u_1,u_2,v_1,v_2\in (\Om BS)^I$,
  such that
  \begin{equation*}
  \varepsilon_j=((\psi(u_1),\psi(au_2)),a,(\psi(u_1a),\psi(u_2))    
  \end{equation*}
  and
  \begin{equation*}
  \delta_j=((\psi(v_1),\psi(bv_2)),b,(\psi(v_1b),\psi(v_2)).    
\end{equation*}
Note that $\alpha(\varepsilon_j)$ and $\alpha(\delta_j)$
belong to the same strongly connected component of $\Gamma_{\psi}$, by the minimality of the index $j$.

  Since $\pv V\models u_1a\cdot u_2=v_1\cdot bv_2$ and $\pv V$ is equidivisible,
  there is $t\in (\Om BS)^I$ such that
  \begin{equation}\label{eq:characterization-1}
    \pv V\models u_1at=v_1\quad \text{and}\quad\pv V\models u_2=tbv_2,
  \end{equation}
  or
  \begin{equation}\label{eq:characterization-2}
    \pv V\models v_1t=u_1a\quad \text{and}\quad\pv V\models bv_2=tu_2.
  \end{equation}
  
  If Case~\eqref{eq:characterization-1} holds, then there is
  in $\Gamma_\psi$ a (possibly empty) path from
  $\omega(\varepsilon_j)$
  to
  $\alpha(\delta_j)$, labeled by a word
  $t_0\in B^\ast$ such that $\psi(t_0)=\psi(t)$.
  Since $\alpha(\delta_j)$ and
  $\omega(\delta_j)$ belong to the same strongly connected component, we conclude there is in $\Gamma_\psi$
  a path from $\omega(\varepsilon_j)$
  to $\alpha(\varepsilon_j)$,
  contradicting the fact that $\varepsilon_j$ is a transition edge.

  Therefore, Case~\eqref{eq:characterization-2} holds
  with $t\neq I$.
  By Proposition~\ref{p:equid-conca-are-finitely-cancelable},
  it follows from~\eqref{eq:characterization-2}
  that there is $t'\in(\Om BS)^I$ with $t=t'a$
  and
  \begin{equation}\label{eq:characterization-3}
    \pv V\models v_1t'=u_1\quad \text{and}\quad\pv V\models bv_2=t'au_2.
  \end{equation}
  
  Suppose that $t'\neq I$. Again by Proposition~\ref{p:equid-conca-are-finitely-cancelable}, it follows from~\eqref{eq:characterization-3}
  that 
  there is $t''\in(\Om BS)^I$ with $t'=bt''$
  and
  \begin{equation}\label{eq:characterization-4}
    \pv V\models v_1b\cdot t''=u_1\quad \text{and}\quad\pv V\models v_2=t''\cdot au_2.
  \end{equation}
  This implies the existence of a path in $\Gamma_{\psi}$
  from $\omega(\delta_j)$ to $\alpha(\varepsilon_j)$, which once more leads to a contradiction with the definition of a transition edge.

  Therefore, we have $t'=I$, and so,
  thanks to Proposition~\ref{p:equid-conca-are-finitely-cancelable},
  from~\eqref{eq:characterization-3}
  we get $\pv V\models v_1=u_1$, $a=b$ and $\pv V\models v_2=u_2$.
  This yields $\varepsilon_j=\delta_j$, which contradicts the initial
  assumption.
  Therefore, the set~\eqref{eq:characterization-0} is empty.
  In particular, $\varepsilon_n=\delta_n$ holds.
  Since $\varepsilon_n$ is the last transition edge for $u$,
  we have $\omega(\delta_n)=(\psi(u),I)$, which means that
  $\delta_n$ is the last transition edge for $v$,
  whence $m=n$ and $\varepsilon_i=\delta_i$ for every $i\in\{1,\ldots,n\}$.
  This concludes the proof
  that $\psi^\krho(u)=\psi^\krho(v)$.

  Now, let $\zeta$ be an arbitrary continuous homomorphism
  from $\Om BS$ into $S_\psi^\krho$. Because $\psi^\krho$ is onto, there is a continuous endomorphism $\lambda$ of $\Om BS$ such that $\zeta=\psi^\krho\circ\lambda$.
  Since we also have $\pv V\models \lambda(u)=\lambda(v)$,
  we deduce that $\zeta(u)=\zeta(v)$. This
  establishes our claim that $S_\psi^\krho\models u=v$.
  
    Applying Lemma~\ref{l:letter-augmentation},
    we conclude that $S_\varphi^\krho\models u=v$.
  By Reiterman's  Theorem~\cite{Reiterman:1982}, we deduce that
  $S_\varphi^\krho\in\pv V$, 
  thus proving that $\pv V$
  is closed under the two-sided Karnofsky-Rhodes expansion.
\end{proof}

\begin{Cor}\label{c:characterization-of-LImV}
  Let $\pv V$ be a pseudovariety of semigroups. The following conditions
  are equivalent:
  \begin{enumerate}
  \item $\pv V$ is equidivisible and it is not contained in $\pv {CS}$;\label{item:characterization-of-LImV-1}    
  \item $\pv V=\pv {LI}\malcev \pv V$;\label{item:characterization-of-LImV-2}
  \item $\pv V$ is closed under the two-sided Karnofsky-Rhodes expansion;\label{item:characterization-of-LImV-3}
  \item $\pv V$ is closed under the two-sided connected expansion.\label{item:characterization-of-LImV-4}
  \end{enumerate}
\end{Cor}

\begin{proof}
  The equivalence
  \eqref{item:characterization-of-LImV-2}$\Leftrightarrow$\eqref{item:characterization-of-LImV-3}
  is Corollary~\ref{c:closure-under-Karnofsky-Rhodes-expansion}.
  In particular, a pseudovariety closed under the
  the two-sided Karnofsky-Rhodes expansion
  contains~$\pv {LI}$. As $\pv {LI}$ is not contained in $\pv {CS}$,
  the equivalences
  \eqref{item:characterization-of-LImV-1}$\Leftrightarrow$\eqref{item:characterization-of-LImV-2}$\Leftrightarrow$\eqref{item:characterization-of-LImV-3} then follow
  from
  Theorem~\ref{t:complete-characterization-of-equidivisible-pseudovarieties}.
  Since the two-sided connected expansion is a quotient of the
  two-sided Karnofsky-Rhodes expansion, we clearly have
  \eqref{item:characterization-of-LImV-3}$\Rightarrow$\eqref{item:characterization-of-LImV-4}.
  Conversely, suppose that  $\pv V$ is closed
  under the two-sided connected expansion.
  By Proposition~\ref{p:equidivisibility-under-connected-expansion},
  $\pv V$ is equidivisible.
  We claim that $\pv V$ is not contained in~$\pv {CS}$.
  Consider the mapping $\varphi$
  from $\Om {\{a\}}S$ onto the trivial semigroup $S=\{1\}$.
  The path in $\Gamma_\varphi$ from $(I,1)$ to $(1,I)$
  labeled $a$ intersects precisely two strongly connected components of
  $\Gamma_\varphi$, while, for every $k\geq 2$, the
  path in $\Gamma_\varphi$ from $(I,1)$ to $(1,I)$
  labeled $a^{k}$ intersects precisely three strongly connected components
  of $\Gamma_\varphi$. Therefore,
  denoting by $\varphi^C$ the canonical continuous homomorphism
  from $\Om {\{a\}}S$ to $S^C$ whose restriction to ${\{a\}}^+$
  is the connected expansion of $\varphi|_{{\{a\}}^+}$, we
  have  $\varphi^C(a)\neq\varphi^C(a^{\omega+1})$.
  This shows that $S^C\notin\pv {CR}$, establishing the claim, and concluding the proof that we have 
  \eqref{item:characterization-of-LImV-4}$\Rightarrow$\eqref{item:characterization-of-LImV-1}.
\end{proof}

There is one important further connection of the conditions of
Corollary~\ref{c:characterization-of-LImV} with varieties of
languages. Indeed, as has been proved by Pin~\cite{Pin:1980b} (cf.\
\cite[Theorem~7.3]{Pin:1997}), the language counterpart of the
operator $\pv V\mapsto\pv{LI}\malcev\pv V$ is the closure under
unambiguous product.
  
We conclude the paper with one further application of our results for
pseudovarieties of local groups. Combining
Corollary~\ref{c:characterization-of-LImV} with
Theorem~\ref{t:equidiv-in-LG}, we obtain the following.
  
\begin{Cor}\label{c:subLG-ovLI}
  If $\pv V$ is a subpseudovariety of\/ $\pv {LG}$ containing $\pv {LI}$
  then $\pv V=\pv {LI}\malcev \pv V$.\qed
\end{Cor}

It is well known that $\pv {LI}\vee \pv H=\pv {LI}\malcev \pv H$ for
every pseudovariety $\pv H$ of groups
\cite[Corollary~3.2]{Hall&Weil:1999}.
The previous results provide an indirect proof of the following
extension of~that fact, which appears to be new.

\begin{Cor}\label{c:CSveeLI}
  If $\pv V$ is a subpseudovariety of\/ $\pv {LG}$
  then $\pv {LI}\vee \pv V=\pv {LI}\malcev \pv V$.
\end{Cor}

\begin{proof}
  Let $\pv V$ be a subpseudovariety of $\pv {LG}$.
  Then the pseudovariety $\pv {LI}\vee \pv V$
  is equidivisible by Theorem~\ref{t:equidiv-in-LG}.
  Therefore, applying Corollary~\ref{c:subLG-ovLI}, we obtain
  $\pv {LI}\vee \pv V
  =\pv {LI}\malcev (\pv {LI}\vee \pv V)\supseteq
  \pv {LI}\malcev \pv V
  \supseteq \pv {LI}\vee \pv V
  $.
\end{proof}

Reading \cite[Corollary~4.3]{Costa:2002a} one finds the
following basis for
$\pv {LI}\vee \pv {CS}$:
\begin{equation*}
  \pv {LI}\vee \pv {CS}=
  \op z^\omega (xy)^\omega x t^\omega = z^\omega x t^\omega,xy^\omega z
  =(xy^\omega z)^{\omega+1}\cl.
\end{equation*}
As an example of application of Corollary~\ref{c:CSveeLI}, we
obtain a simplified basis for $\pv {LI}\vee \pv {CS}$,
made of a pseudoidentity involving only
three letters.

\begin{Prop}\label{p:LIveeCS}
  The pseudovariety $\pv{LI}\vee\pv{CS}$ is defined by the
  pseudoidentity $(xy)^\omega(xz)^\omega(xy)^\omega=(xy)^\omega$.
\end{Prop}

\begin{proof}
  By~\cite[Theorem~6.1]{Almeida&Margolis&Steinberg&Volkov:2010}, taking
  $\Sigma=\{(zt)^\omega z=z\}$, $\pv H=\pv I$, $W=\{x_1\}$, $m=0$, and
  $\alpha_1=x_1^{\omega+1}$, we obtain that the pseudovariety
  $\pv{LI}\malcev\pv{CS}$ is defined by the following pseudoidentities:
  \begin{align}
    \label{eq:LImCS-1}
    \Bigl(
    \bigl(x(zt)^\omega zy\bigr)^\omega
    xzy
    \bigl(x(zt)^\omega zy\bigr)^\omega
    \Bigr)^\omega
    &=
      \bigl(x(zt)^\omega zy\bigr)^\omega
    \\
    \label{eq:LImCS-2}
    (xzy)^{\omega-1}
    \big(x(zt)^\omega zy\bigr)^{\omega+1}
    (xzy)^\omega
    &=
      (xzy)^\omega.
  \end{align}
  Note that the first of these pseudoidentities is valid in~\pv{LG}
  while the pseudovariety defined by the second one is contained
  in~\pv{LG}. Hence, $\pv{LI}\malcev\pv{CS}$ is defined by the
  pseudoidentity \eqref{eq:LImCS-2}. Further simplifications of the
  pseudoidentity \eqref{eq:LImCS-2} may be carried out as follows.
  First, since it defines a subpseudovariety of~\pv{LG}, the
  $(\omega+1)$-power of the infinite element in the middle may be
  replaced by the base of that power. Second, pre-multiplying both
  sides by $zy$, post-multiplying by $x(zyx)^{\omega-1}$ and applying
  suitable conjugations to shift infinite powers, we obtain the
  pseudoidentity $(zyx)^\omega(zt)^\omega(zyx)^\omega=(zyx)^\omega$.
  Substituing $x$ by $zy$, we deduce the equivalent pseudoidentity
  $(zy)^\omega(zt)^\omega(zy)^\omega=(zy)^\omega$, as the former can
  be recovered by substituting in the latter $y$ by $yx$. Renaming
  variables, this shows that
  $(xy)^\omega(xz)^\omega(xy)^\omega=(xy)^\omega$ is a simplified
  basis for $\pv{LI}\malcev\pv{CS}$. Finally, we apply
  Corollary~\ref{c:CSveeLI}.
\end{proof}

\bibliographystyle{amsplain}
\providecommand{\bysame}{\leavevmode\hbox to3em{\hrulefill}\thinspace}
\providecommand{\MR}{\relax\ifhmode\unskip\space\fi MR }
% \MRhref is called by the amsart/book/proc definition of \MR.
\providecommand{\MRhref}[2]{%
  \href{http://www.ams.org/mathscinet-getitem?mr=#1}{#2}
}
\providecommand{\href}[2]{#2}

\end{document}